\documentclass[12pt,oneside,english]{amsart}
\usepackage[T1]{fontenc}
\usepackage[latin9]{inputenc}
\usepackage{geometry}
\usepackage{color}
\usepackage{amsthm}
\usepackage{amssymb}

\makeatletter
\numberwithin{equation}{section}
\numberwithin{figure}{section}
\theoremstyle{plain}
\newtheorem{thm}{\protect\theoremname}[section]
  \theoremstyle{remark}
  \newtheorem*{rem*}{\protect\remarkname}
  \theoremstyle{plain}
  \newtheorem{lem}[thm]{\protect\lemmaname}
  \theoremstyle{plain}
  \newtheorem{cor}[thm]{\protect\corollaryname}
  \theoremstyle{plain}
  \newtheorem{prop}[thm]{\protect\propositionname}
  \theoremstyle{plain}
  \newtheorem*{thm*}{\protect\theoremname}
  \theoremstyle{plain}
  \newtheorem*{prop*}{\protect\propositionname}

\makeatother

\usepackage{babel}
  \providecommand{\corollaryname}{Corollary}
  \providecommand{\lemmaname}{Lemma}
  \providecommand{\propositionname}{Proposition}
  \providecommand{\remarkname}{Remark}
  \providecommand{\theoremname}{Theorem}
\providecommand{\theoremname}{Theorem}

\begin{document}
\sloppy

\title[Differentiation operator from model spaces to Bergman spaces]{Differentiation operator 
from model spaces 
\\ to Bergman spaces and Peller type inequalities}

\author{Anton Baranov}

\address{Department of Mathematics and Mechanics, Saint Petersburg State University,
28, Universitetski pr., St. Petersburg, 198504, Russia.}

\email{anton.d.baranov@gmail.com}

\author{Rachid Zarouf}

\address{Aix-Marseille Universit\'e, CNRS, Centrale Marseille, I2M, UMR 7373,
13453 Marseille, France.}

\address{\'ESP\'E (\'Ecole Sup\'erieure du Professorat et de l'\'Education) d'Aix-Marseille,  Aix-Marseille Universit\'e, 32, Rue Eug\`ene Cas 13248 Marseille Cedex 04, France.}

\address{Department of Mathematics and Mechanics, Saint Petersburg State University,
28, Universitetski pr., St. Petersburg, 198504, Russia.}

\email{rachid.zarouf@univ-amu.fr}

\thanks{The work is supported by Russian Science Foundation grant 14-41-00010.}


\begin{abstract}
Given an inner function $\Theta$ in the unit disc $\mathbb{D}$,
we study the boundedness of the differentiation operator which acts from
the model subspace $K_{\Theta}=\left(\Theta H^{2}\right)^{\perp}$ of
the Hardy space $H^{2},$ equiped with the $BMOA$-norm, to some radial-weighted
Bergman space. As an application, we generalize Peller's inequality
for Besov norms of rational functions $f$ of degree $n\geq1$ having
no poles in the closed unit disc $\overline{\mathbb{D}}$. 
\end{abstract}
\maketitle

\section{Introduction and notations}

A well-known inequality by Vladimir Peller (see inequality \eqref{Peller_estimate} below)
majorizes a Besov norm of any rational function $f$ of degree $n\geq1$
having no poles in the closed unit disc
$\overline{\mathbb{D}}=\left\{ \xi\in \mathbb{C}:\;|\xi|\leq1\right\} $
in terms of its $BMOA$-norm and its degree $n$. The original proof
of Peller is based on his description 
of Hankel operators in the Schatten classes. One of the
aims of this paper is to give a short and direct proof of this inequality and
extend it to more general radial-weighted Bergman norms. Our proof
combines integral representation for the derivative of $f$ 
(which come from the  theory of model spaces) 
and the generalization of a theorem by E.M. Dyn'kin.
The corresponding inequalities are obtained in terms of radial-weighted
Bergman norms of the derivative of finite Blaschke products (of degree
$n=\deg f$), instead of $n$ itself. The finite Blaschke products
in question have the same poles as $f$. The study of radial-weighted
Bergman norms of the derivatives of finite Blaschke products of degree
$n$ and their asymptotic as $n$ tends to $+\infty$ is of independent
interest. A contribution to this topic, which we are going to exploit
here, was given by J.~Arazy, S.D.~Fisher and J.~Peetre.
\medskip

Let $\mathcal{P}_{n}$ be the space of complex analytic polynomials
of degree at most $n$ and let
\[
\mathcal{R}_{n}^{+}=\left\{ \frac{P}{Q}\,:\; P,\, 
Q\in\mathcal{P}_{n},\; Q(\xi)\ne0\text{ for }|\xi|\leq1\right\}
\]
be the set of rational functions of degree at most $n$ with poles
outside of the closed unit disc $\overline{\mathbb{D}}.$
In this paper, we consider the norm of a rational function $f\in\mathcal{R}_{n}^{+}$
in different spaces of analytic functions
in the open unit disc $\mathbb{D}=\left\{ \xi:\;|\xi|<1\right\} $.

\subsection{Some Banach spaces of analytic functions}

We denote by $\mathcal{H}ol(\mathbb{D})$ the space of all holomorphic
functions in $\mathbb{D}.$

\subsubsection{The Besov spaces $B_{p}$}
A function $f\in\mathcal{H}ol(\mathbb{D})$ belongs to the Besov space
$B_{p},$ $1<p < \infty$, if and only if
\[
\left\Vert f\right\Vert _{B_{p}}=\vert f(0)\vert+\left\Vert f\right\Vert _{B_{p}}^{\star}<+\infty,
\]
where $\left\Vert f\right\Vert _{B_{p}}^{\star}$ is the seminorm
defined by
\[
\left\Vert f\right\Vert _{B_{p}}^{\star}=
\left(\int_{\mathbb{D}}(1-|u|^2)^{p-2}
\left|f'(u)\right|^{p}{\rm d}A(u)\right)^{\frac{1}{p}},
\]
$A$ being the normalized area measure on $\mathbb{D}$.

For the case $0<p\le 1$ the definition of the Besov norm requires a modification:
\[
\left\Vert f\right\Vert _{B_{p}} =
\sum_{j=0}^{k-1} |f^{(j)}(0)| + 
\left\Vert f\right\Vert _{B_{p}}^{\star},
\qquad
\left\Vert f\right\Vert _{B_{p}}^{\star} =  
\bigg(\int_{\mathbb{D}}\left|f^{(k)}(u)\right|^p (1-|u|^2)^{pk-2}{\rm d}A(u)
\bigg)^{\frac{1}{p}},
\]
where $k$ is the smallest positive integer such that $pk>1$.
We refer to \cite{Pee,Tri,BeLo} for general properties of Besov spaces.

A function $f\in\mathcal{H}ol(\mathbb{D})$ belongs to the space $B_{\infty}$
(known as the Bloch space) if and only if  
$\left\Vert f\right\Vert _{B_{\infty}}=
\left|f(0)\right|+\sup_{z\in\mathbb{D}}\left|f'(z)\right|\left(1-\left|z\right|\right)<\infty$.

\subsubsection{The radial-weighted Bergman spaces
$A_{p}\left(w\right)$}

The radial-weighted Bergman space $A_{p}\left(w\right)$, $1\leq p<\infty$,
is defined as:
\[
A_{p}\left(w\right)=\left\{ f\in\mathcal{H}ol\left(\mathbb{D}\right):\:\left\Vert f\right\Vert _{A_{p}\left(w\right)}^{p}=\int_{\mathbb{D}}w(\vert u\vert)\left|f(u)\right|^{p}{\rm d}A(u)<\infty\right\} ,
\]
where the weight $w$ satisfies $w\geq0$ and $\int_{0}^{1}w(r)\,{\rm d}r<\infty.$
The classical power weights $w(r)=w_{\alpha}(r)=\left(1-r^{2}\right)^{\alpha}$,
$\alpha>-1$, are of special interest; in this case we 
put $A_{p}(\alpha)=A_{p}\left(w_{\alpha}\right).$
We refer to \cite{HKZ} for general properties of weighted Bergman
spaces.

\subsubsection{The spaces $A_{p}^{1}(\alpha)$}

A function $f\in\mathcal{H}ol(\mathbb{D})$ belongs to the space $A_{p}^{1}(\alpha),$
$1\leq p\leq+\infty$, $\alpha>-1,$ if and only if
\[
\left\Vert f\right\Vert _{A_{p}^{1}(\alpha)}=\vert f(0)\vert+\left\Vert f'\right\Vert _{A_{p}\left(\alpha\right)}<+\infty.
\]
We also define the $A_{p}^{1}(\alpha)$-seminorm by $\left\Vert f\right\Vert _{A_{p}^{1}(\alpha)}^{\star}=\left\Vert f'\right\Vert _{A_{p}\left(\alpha\right)}.$
Note that the spaces $B_{p}$ and $A_{p}^{1}(p-2)$ coincide for $1<p < +\infty.$

\subsubsection{The space BMOA }

There are many ways to define $BMOA$; see \cite[Chapter 6]{Gar}.
For the purposes of this paper we choose the following one: a function
$f\in\mathcal{H}ol(\mathbb{D})$ belongs to the $BMOA$ space (of
analytic functions of bounded mean oscillation) if and only if
\[
\left\Vert f\right\Vert _{BMOA}=\inf\left\Vert g\right\Vert _{L^{\infty}(\mathbb{T})}<+\infty,
\]
where the infimum is taken over all $g\in L^{\infty}(\mathbb{T})$,
$\mathbb{T}=\left\{\xi:\;|\xi|=1\right\}$ being
the unit circle, for which the representation
\[
f(\xi)=\frac{1}{2\pi i}\int_{\mathbb{T}}\frac{g(u)}{u-\xi}{\rm d}u,\qquad\vert\xi\vert<1,
\]
holds. Recall that $BMOA$ is the dual space of the
Hardy space $H^{1}$ under the pairing
\[
\left\langle f,\, g\right\rangle =\int_{\mathbb{T}}f(u)\overline{g(u)}{\rm d}u,\qquad f\in H^{1},\: g\in BMOA,
\]
where this integral must be understood as the extension of the pairing
acting on a dense subclass of $H^{1},$ see \cite[p. 23]{Bae}.

\subsection{\label{def_model_space}Model spaces }

\subsubsection{\label{sub:General-inner-functions}General inner functions}

By $H^{p},\,1\leq p\leq\infty,$ we denote the standard Hardy spaces
(see \cite{Gar,Nik}). Recall that $H^{2}$ is a reproducing kernel
Hilbert space, with the kernel
\[
k_{\lambda}(w)=\frac{1}{1-\overline{\lambda}w},\qquad\lambda,\, w\in\mathbb{D},
\]
 known as the Szeg\"{o} kernel (or the Cauchy kernel) associated with
$\lambda.$ Thus $\left\langle f,\, k_{\lambda}\right\rangle =f(\lambda)$
for all $f\in H^{2}$ and for all $\lambda\in\mathbb{D}$, where
$\left\langle \cdot,\,\cdot\right\rangle $
is the scalar product on $H^{2}.$

Let $\Theta$ be an \textit{inner function}, i.e., $\Theta\in H^{\infty}$
and $\vert\Theta(\xi)|=1$ a.e. $\xi\in\mathbb{T}$.
We define the model subspace $K_{\Theta}$ of the Hardy space $H^{2}$
by
\[
K_{\Theta}=\left(\Theta H^{2}\right)^{\perp}=H^{2}\ominus\Theta H^{2}.
\]
By the famous theorem of Beurling, these and only these subspaces
of $H^{2}$ are invariant with respect to the backward shift operator.
We refer to \cite{Nik} for the general theory of the spaces $K_{\Theta}$
and their numerous applications.

For any \textit{inner function} $\Theta$, the reproducing kernel
of the model space $K_{\Theta}$ corresponding to a point $\xi\in\mathbb{D}$
is of the form
\[
k_{\lambda}^{\Theta}(w)=\frac{1-\overline{\Theta(\lambda)}\Theta(w)}
{1-\overline{\lambda}w},\qquad \lambda,\, w\in\mathbb{D},
\]
that is $\left\langle f,\, k_{\lambda}^{\Theta}\right\rangle =f(\lambda)$
for all $f\in K_{\Theta}$ and for all $\lambda\in\mathbb{D}$.

\subsubsection{\label{sub:def_of_K_B}The case of finite Blaschke products }

From now on, for any $\sigma=(\lambda_{1},\dots,\lambda_{n})\in\mathbb{D}^{n}$,
we consider the finite Blaschke product
\[
B_{\sigma}=\prod_{k=1}^{n}b_{\lambda_{k}},
\]
where $b_{\lambda}(z)=\frac{\lambda-z}{1-\overline{\lambda}z}$, is
the elementary Blaschke factor corresponding to $\lambda\in\mathbb{D}$.
It is well known that if 
\[
\sigma=\{\lambda_{1},...,\lambda_{1},\lambda_{2},...,
\lambda_{2},...,\lambda_{t},...,\lambda_{t}\}\in\mathbb{D}^{n},
\]
where every $\lambda_{s}$ is repeated according to its multiplicity
$n_{s}$, $\sum_{s=1}^{t}n_{s}=n$, then
\[
K_{B_{\sigma}}=H^{2}\ominus B_{\sigma}H^{2}=
\overline{{\rm span}} 
\{ k_{\lambda_{j},\, i}:\,1\leq j\leq t,\,0\leq i\leq n_{j}-1 \},
\]
where for $\lambda\neq0$, $k_{\lambda,\, i}=\left(\frac{d}{d\overline{\lambda}}\right)^{i}k_{\lambda}$
and $k_{\lambda}=\frac{1}{1-\overline{\lambda}z}$ is the standard
Cauchy kernel at the point $\lambda,$ whereas $k_{0,\, i}=z^{i}.$
Thus the subspace $K_{B_{\sigma}}$ consists of rational functions of the
form $p/q$, where $p\in\mathcal{P}_{n-1}$ and $q\in\mathcal{P}_{n}$,
with the poles $1/\overline{\lambda}_{1},\dots,1/\overline{\lambda}_{n}$
of corresponding multiplicities (including possible poles at $\infty$).
Hence, if $f\in\mathcal{R}_{n}^{+}$ and 
$1/\overline{\lambda}_{1},\dots,1/\overline{\lambda}_{n}$
are the poles of $f$ (repeated according to multiplicities), then
$f\in K_{zB_{\sigma}}$ with $\sigma=(\lambda_{1},\dots,\lambda_{n})$.
\medskip

From now on, for two positive functions $a$ and $b$, we say that
$a$ is dominated by $b$, denoted by $a\lesssim b$, if there is
a constant $C>0$ such that $a\leq Cb$; we say that $a$ and
$b$ are comparable, denoted by $a\asymp b$, if both $a\lesssim b$
and $b\lesssim a$.

\medskip


\section{\label{sub:Known-results-and}Main results}

\subsection{\label{sub:Main-ingredients}Main ingredients}

In 1980 V. Peller proved in his seminal paper \cite{Pel1} that
\begin{equation}
\label{Peller_estimate}
\left\Vert f\right\Vert _{B_{p}} \leq c_{p}
n^{\frac{1}{p}}\left\Vert f\right\Vert _{BMOA}
\end{equation}
for any $f\in\mathcal{R}_{n}^{+}$ and $1\leq p\leq+\infty,$ where
$c_{p}$ is a constant depending only on $p.$ Later, this result
was extended to the range $p>0$ independently and with different
proofs by Peller \cite{Pel2}, S. Semmes \cite{Sem} and also by
A. Pekarskii  \cite{pek1} who found a proof which does not use
the theory of Hankel operators (see also \cite{pek2}).
 
The aim of the present article is:
\smallskip

\begin{enumerate}
\item study the boundedness of the differention operator from 
$\left(K_{\Theta},\,\left\Vert \cdot\right\Vert _{BMOA}\right)$
to $A_{p}\left(\alpha\right),$ $1<p < +\infty,$ $\alpha>-1,$ and
\smallskip

\item generalize Peller's result (\ref{Peller_estimate}) replacing the
$B_{p}$-seminorm by the $A_{p}^{1}(\alpha)$-one.
\end{enumerate}
\smallskip
In both of these problems, we make use of a method based on two main
ingredients:
\smallskip
\begin{itemize}
\item integral representation for the derivative of functions in $K_{\Theta}$
or in $\mathcal{R}_{n}^{+}$, and
\smallskip

\item a generalization of a theorem by E.M.  Dyn'kin, see Subsection \ref{Subsec_Dynkin_gen}.
\end{itemize}
\smallskip

One more tool (that we will need in problem (2)) is the estimate
of $B_{p}$-seminorms of finite Blaschke products by Arazy, Fischer
and Peetre \cite{AFP}.


\subsection{Main results}

Let us consider the differentation operator $Df=f'$ and the shift
and the backward shift operators defined respectively by
\begin{equation}
Sf=zf,\qquad S^{\star}f=\frac{f-f(0)}{z},\label{diff_and_shift_op}
\end{equation}
for any $f\in\mathcal{H}ol(\mathbb{D}).$ From now on, for any inner
function $\Theta,$ we put
\[
\widetilde{\Theta}=z\Theta=S\Theta.
\]

\subsubsection{\label{norm_K_Theta}Boundedness of the differentiation operator
from $\left(K_{\Theta},\,\left\Vert \cdot\right\Vert _{BMOA}\right)$
to $A_{p}\left(\alpha\right)$}

Let us first discuss the
boundedness of the operator $D$ from $BMOA$ to $A_{p}\left(\alpha\right)$.
The following (essentially well-known) proposition gives 
necessary and sufficient conditions on $p$ and
$\alpha$ so that a continuous embedding $BMOA\subset A_{p}^{1}\left(\alpha\right)$ hold.

\begin{prop}
\label{prop:embed}
Let $\alpha>-1$ and $1\le p <\infty$. Then
$BMOA\subset A_{p}^{1}\left(\alpha\right)$ if and only if either $\alpha>p-1$
or $\alpha=p-1$ and $p\ge 2$.
\end{prop}

Now, we consider an arbitrary \textit{inner function} $\Theta$. Our
first main result gives necessary and sufficient conditions under
which the differentiation operator
\[
D:\:\left(K_{\Theta},\,\left\Vert \cdot\right\Vert _{BMOA}\right)\rightarrow
A_{p}\left(\alpha\right)
\]
is bounded.
When this is the case, we estimate its norm in terms of $\|\Theta'\|_{A_{p}\left(\alpha\right)}$.

\begin{thm}
\label{thm1} 
Let $1<p < \infty$ and $\alpha>-1$.
Then the operator $D:\:\left(K_{\Theta},\,\left\Vert 
\cdot\right\Vert _{BMOA}\right)\rightarrow A_{p}\left(\alpha\right)$
is bounded if and only if $\Theta' \in A_{p}\left(\alpha\right)$.

Moreover, one can distinguish three cases:

$(a)$ If $\alpha>p-1$ or $\alpha = p-1$, and $p\ge 2$ then the operator 
$D:\:\left(K_{\Theta},\,\left\Vert 
\cdot\right\Vert _{BMOA}\right)\rightarrow A_{p}\left(\alpha\right)$
is bounded.

$(b)$ If $p-2<\alpha<p-1$, then the operator $D:\:\left(K_{\Theta},\,\left\Vert 
\cdot\right\Vert _{BMOA}\right)\rightarrow A_{p}\left(\alpha\right)$
is bounded if and only if $\Theta'\in A_{1}\left(\alpha-p+1\right).$

$(c)$ If $\alpha \le p-2$, then the operator 
$D:\:\left(K_{\Theta},\,\left\Vert \cdot\right\Vert _{BMOA}\right)\rightarrow A_{p}\left(\alpha\right)$
is bounded if and only if $\Theta$ is a finite Blaschke product.

In cases $(b)$ and $(c)$, we have
\begin{equation}
\left\Vert D\right\Vert \lesssim\|\Theta'\|_{A_{p}\left(\alpha\right)}\lesssim\left\Vert D\right\Vert +{\rm const},\label{norm_D_K_Theta}
\end{equation}
with constants depending on $p$ and $\alpha$ only. 
\end{thm}

\begin{rem*}
1. In the cases $(b)$ and $(c)$, to show that their conditions are equivalent 
to the inclusion $\Theta'\in A_{p}(\alpha)$ we use a theorem by
P.R. Ahern \cite{Ahe1} and its generalizations by I.E.~Verbitsky \cite{Ver}
and A. Gluchoff \cite{Glu}. We do not know whether
the inclusions $\Theta'\in A_{p}(\alpha)$ and $\Theta'\in A_{1}\left(\alpha-p+1\right)$
are equivalent for $\alpha=p-1$ and $ 1 < p<2$.     
\smallskip

2. The membership of Blaschke products in
various function spaces is a well-studied topic. Besides the above-cited
papers by Ahern, Gluchoff and Verbitsky, let us mention the papers
by Ahern and D.N. Clark
\cite{AC1, AC2} and recent works by D. Girela, J. Pel\'aez, D. Vukoti\'c, and A. Aleman
\cite{GPV, AV}.
\end{rem*}

\subsubsection{\label{Subsec_Peller_gen}Generalization of Peller's inequalities}

In the following theorem, we give a generalization of Peller's inequality
\eqref{Peller_estimate}.

\begin{thm}
\label{thm2} Let $f\in\mathcal{R}_{n}^{+},$ $\deg f=n$ and $\sigma\in\mathbb{D}^{n}$
be the set of its poles counting multiplicities \textup(including poles at
$\infty$\textup). For any $\alpha>-1,$ $1<p < \infty$, and $p>1+\alpha$,
we have
\begin{equation}
\left\Vert f\right\Vert _{A_{p}^{1}(\alpha)}^{\star}\leq K_{p,\,\alpha}
\left\Vert f\right\Vert _{BMOA}\|B_{\sigma}'\|_{A_{p}\left(\alpha\right)},
\label{right-hand-side-peller-gen-ineq}
\end{equation}
where $K_{p,\,\alpha}^{p}=\frac{2^{\frac{1}{\alpha+1}}}{2^{\frac{1}{\alpha+1}}-1}
\left(\frac{p}{p-1-\alpha}\right)^{p}2^{p+1}.$ 
\end{thm}

\begin{rem*}
$\,$The inequality \eqref{right-hand-side-peller-gen-ineq} is sharp 
up to a constant in the following sense:  for $f=B_\sigma$ 
we, obviously, have $\|f\|_{A_p^1(\alpha)}^{\star} = \|f\|_{BMOA} 
\|B_\sigma'\|_{A_p(\alpha)}$ (note that $\|B_\sigma\|_{BMOA} =1$). 
\medskip

Let us show how Peller's inequality \eqref{Peller_estimate}
for $1<p<\infty$ follows from Theorem \ref{thm2}. For $\alpha=p-2$, we have
\[
\left\Vert f'\right\Vert _{A_{p}\left(\alpha\right)}=
\left\Vert f\right\Vert _{B_{p}}^{\star},\qquad\|B_{\sigma}'\|_{A_{p}
\left(\alpha\right)}=\|B_{\sigma}\|_{B_{p}}^{\star}.
\]
To deduce Peller's inequalities  it remains
to apply the following theorem by Arazy, Fischer and Peetre \cite{AFP}:
if $1\leq p\leq\infty$, then there exist absolute positive constants
$m_{p}$ and $M_{p}$ such that
\begin{equation}
m_{p}n^{\frac{1}{p}}\leq\| B\|_{B_{p}}^{\star}\leq M_{p}n^{\frac{1}{p}}.
\label{eq:AFP_inequ}
\end{equation}
for any Blaschke product of degree $n$. Then we obtain for $1<p < \infty$,
\[
\left\Vert f\right\Vert _{B_{p}}^{\star}\leq K_{p,\, p-2}^{\frac{1}{p}}M_{p}\left\Vert f\right\Vert _{BMOA}(n+1)^{\frac{1}{p}}\lesssim n^{\frac{1}{p}}\left\Vert f\right\Vert _{BMOA}.
\]

To make the expositions self-contained, we give in Section \ref{peller}
a very simple proof of the upper estimate in 
\eqref{eq:AFP_inequ} (which is slightly different from the proof 
by D. Marshall presented in \cite{AFP}).

The method of integral repesentations for higher order derivatives in model spaces
allows to prove Peller's inequalities also for $0<p\le 1$.  
In Section \ref{pel1} we present the proof for the case $p>\frac{1}{2}$. 
\end{rem*}

\subsubsection{\label{Subsec_Dynkin_gen}Generalization of a theorem by Dyn'kin}

E.M. Dyn'kin proved in \cite[Theorem 3.2]{Dyn} that \textit{
\begin{equation}
\label{Dyn'kin_estimate}
\int_{\mathbb{D}}\left(\frac{1-\vert B(u)\vert^{2}}{1-\left|u\right|^{2}}\right)^{2}{\rm d}A(u)\leq8(n+1),
\end{equation}
for any finite Blaschke product $B$ of degree $n$. }

From now on, for any\textit{ inner function} $\Theta$ and for any
$\alpha>-1,$ $p>1,$ we put                          
\begin{equation}
I_{p,\,\alpha}(\Theta)=\int_{\mathbb{D}}(1-\vert u\vert^{2})^{\alpha}\left(\frac{1-\vert\Theta(u)\vert^{2}}{1-\left|u\right|^{2}}\right)^{p}{\rm d}A(u).\label{def_I_B_gen}
\end{equation}
Dyn'kin's Theorem can be stated as follows: \textit{for any finite
Blaschke product $B$ of degree $n,$ we have
\[
I_{2,\,0}(B)\leq8(n+1).
\]
}

Here, we generalize this result to the case $\alpha>-1,$ $p>1$
and $p>1+\alpha$. This generalization is the key step of 
the proof of Theorem \ref{thm2}.

\begin{thm}
\label{th3} 
Let $1< p <\infty$, $\alpha>-1$ and $p>1+\alpha$.
Then,
\[
\left\Vert \Theta'\right\Vert _{A_{p}\left(\alpha\right)}^{p}\leq I_{p,\,\alpha}(\Theta)\leq K_{p,\,\alpha}\left\Vert \Theta'\right\Vert _{A_{p}\left(\alpha\right)}^{p},
\]
where $K_{p,\,\alpha}$ is the same constant as in Theorem \ref{thm2}.
\end{thm}

The paper is organized as follows.
We first focus in Section \ref{Gen-Dynk-Th}
on the generalization of Dyn'kin's result. In Section \ref{pr21},
Proposition \ref{prop:embed} and Theorem \ref{thm1} are proved, while
Section \ref{peller}
is devoted to the proof
of Peller type inequalities (Theorem \ref{thm2}).
The case $\frac{1}{2} <p \le 1$ in Peller's inequality is considered
in Section \ref{pel1}.
In Section \ref{sec:Radial-weighted-Bergman-norms},
we discuss some estimates of radial-weighted Bergman norms
of Blaschke products.
Finally, in Section \ref{dolg} we discuss some related inequalities
by Dolzhenko for which we give a very simple proof for the case $1\le p\le 2$ based
on Dyn'kin's estimate and suggest a way to extend
these inequalities to the range $p>2$.
\medskip


\section{\label{Gen-Dynk-Th}Generalization of Dyn'kin's Theorem}

The aim of this Section is to prove Theorem \ref{th3}. The lower
bound follows trivially from the Schwarz--Pick
inequality applied to $\Theta$. The main
ideas for the proof of the upper bound come from \cite[Theorem 3.2]{Dyn}. 
In this Section, $\Theta$ is an arbitrary inner function.

\begin{lem}
\label{first_estimate_I_B} For $p>1,$ $\alpha>-1$ and $p>1+\alpha,$
we have
\[
I_{p,\,\alpha}(\Theta)\leq2^{p}\int_{0}^{2\pi}\int_{0}^{1}(1-r)^{\alpha}\left(\frac{1}{1-r}\int_{r}^{1}\vert\Theta'(se^{i\theta})\vert{\rm d}s\right)^{p}{\rm d}r\frac{{\rm d}\theta}{\pi}.
\]
\end{lem}

\begin{proof}
Writing the integral $I_{p,\,\alpha}(\Theta)$ in polar coordinates,
and using the fact that
\[
1-\vert\Theta(u)\vert^{2}\leq2(1-\vert\Theta(u)\vert),
\]
we obtain
\begin{eqnarray*}
I_{p,\,\alpha}(\Theta) & \leq & 2^{p}\int_{0}^{1}r(1-r^{2})^{\alpha-p}\left(\int_{0}^{2\pi}(1-\vert\Theta(re^{i\theta})\vert)^{p}\frac{{\rm d}\theta}{\pi}\right){\rm d}r\\
 & \leq & 2^{p}\int_{0}^{1}r(1-r)^{\alpha-p}\left(\int_{0}^{2\pi}\vert\Theta(e^{i\theta})-\Theta(re^{i\theta})\vert^{p}\frac{{\rm d}\theta}{\pi}\right){\rm d}r\\
 & \leq & 2^{p}\int_{0}^{1}r(1-r)^{\alpha-p}\left(\int_{0}^{2\pi}\left(\int_{r}^{1}\vert\Theta'(se^{i\theta})\vert{\rm d}s\right)^{p}\frac{{\rm d}\theta}{\pi}\right){\rm d}r\\
 & \leq & 2^{p}\int_{0}^{2\pi}\int_{0}^{1}(1-r)^{\alpha}\frac{1}{(1-r)^{p}}\left(\int_{r}^{1}\vert\Theta'(se^{i\theta})\vert{\rm d}s\right)^{p}{\rm d}r\frac{{\rm d}\theta}{\pi},
\end{eqnarray*}
which completes the proof of the lemma.
\end{proof}
We recall now a general version of the Hardy inequality, see \cite[page 245]{HLP},
which after change of variables gives (as in \cite[Lemma 7]{Ahe2}):

\begin{lem}
\label{Hardy_ineq} If $h:(0,\,1)\rightarrow\left[0,+\infty)\right.$,
$p>1,$ $\alpha>-1$ and $p>1+\alpha,$ then
\[
\int_{0}^{1}(1-r)^{\alpha}\left(\frac{1}{1-r}\int_{r}^{1}h(s){\rm d}s\right)^{p}{\rm d}r\leq\left(\frac{p}{p-1-\alpha}\right)^{p}\int_{0}^{1}(1-r)^{\alpha}h(r)^{p}{\rm d}r.
\]
\end{lem}

\begin{cor}
\label{cor:We-suppose-that} Let $p>1,$ $\alpha>-1$ and
$p>1+\alpha.$ Then,
\[
I_{p,\,\alpha}(\Theta)\leq C_{p,\,\alpha}\int_{0}^{2\pi}\int_{0}^{1}(1-r)^{\alpha}\vert\Theta'(re^{i\theta})\vert^{p}{\rm d}r\frac{{\rm d}\theta}{\pi},
\]
where $C_{p,\,\alpha}=\left(\frac{p}{p-1-\alpha}\right)^{p}2^{p}.$ \end{cor}
\begin{proof}
Combining estimates in Lemma \ref{first_estimate_I_B} and Lemma \ref{Hardy_ineq}
(setting $h(s)=h_{\theta}(s)=\vert\Theta'(se^{i\theta})\vert$, for
any fixed $\theta\in [0,\,2\pi)),$ we obtain
\[
\int_{0}^{1}(1-r)^{\alpha}\left(\frac{1}{1-r}\int_{r}^{1}\vert\Theta'(se^{i\theta})\vert{\rm d}s\right)^{p}{\rm d}r\leq\left(\frac{p}{p-1-\alpha}\right)^{p}\int_{0}^{1}(1-r)^{\alpha}\vert\Theta'(re^{i\theta})\vert^{p}{\rm d}r.
\]
Thus,
\[
I_{p,\,\alpha}(\Theta)\leq\left(\frac{p}{p-1-\alpha}\right)^{p}2^{p}\int_{0}^{2\pi}\int_{0}^{1}(1-r)^{\alpha}\vert\Theta'(re^{i\theta})\vert^{p}{\rm d}r\frac{{\rm d}\theta}{\pi}
\]
which completes the proof. \end{proof}
\begin{lem}
\label{weighted_bergman_norm} Let any nonzero weight
$w$ satisfying $w\geq0$ and $\int_{0}^{1}w(r)\,{\rm d}r<\infty.$
Let $\beta=\beta_{w}\in(0,1)$ such that $\int_{0}^{1}w(r)\,\mbox{d}r=2\int_{0}^{\beta}w(r)\,\mbox{d}r.$
Then, for $f\in A_{p}\left(w\right),\:\,1\leq p<\infty,$
\begin{align*}
\left\Vert f\right\Vert _{A_{p}\left(w\right)}^{p} & \leq\int_{0}^{2\pi}w(r)\int_{0}^{1}\vert f(re^{i\theta})\vert^{p}{\rm d}r\frac{{\rm d}\theta}{\pi}\\
 & \leq\frac{2}{\beta}\int_{\beta}^{1}rw(r)\left(\int_{0}^{2\pi}\vert f(re^{i\theta})\vert^{p}\frac{{\rm d}\theta}{\pi}\right){\rm d}r\leq\frac{2}{\beta}\left\Vert f\right\Vert _{A_{p}\left(w\right)}^{p}.
\end{align*}
\end{lem}

\begin{proof}
The proof follows easily from the fact that for any $f$ in
$\mathcal{H}ol\left(\mathbb{D}\right)$,
the function
\[
r\mapsto\int_{0}^{2\pi}\vert f(re^{i\theta})\vert^{p}\frac{{\rm d}\theta}{\pi},
\]
is nondecreasing on $[0,1]$.
\end{proof}

We are now ready to prove Theorem \ref{th3}.

\begin{proof}
We first prove (\ref{right-hand-side-peller-gen-ineq}). Applying
Lemma \ref{weighted_bergman_norm} with $f=\Theta'$
and $w(r)=(1-r^{2})^{\alpha},\,\alpha>-1,$ and Corollary 
\ref{cor:We-suppose-that} we obtain that
\[
I_{p,\,\alpha}(\Theta)\leq C_{p,\,\alpha}\int_{0}^{2\pi}\int_{0}^{1}(1-r)^{\alpha}\vert\Theta'(re^{i\theta})\vert^{p}{\rm d}r\frac{{\rm d}\theta}{\pi}\leq\frac{2}{\beta}C_{p,\,\alpha}\left\Vert \Theta'\right\Vert _{A_{p}\left(\alpha\right)}^{p},
\]
where $C_{p,\,\alpha}=\left(\frac{p}{p-1-\alpha}\right)^{p}2^{p},$
and $\beta=\beta_{\alpha}$ satisfies the condition
\[
\int_{\beta}^{1}w(r)\,\mbox{d}r=\int_{0}^{\beta}w(r)\,\mbox{d}r.
\]
By a direct computation, we see that $\beta=\beta_{\alpha}$ is given
by the equation $\frac{1-(1-\beta)^{\alpha+1}}{1+\alpha}=\frac{(1-\beta)^{\alpha+1}}{1+\alpha},$
which is equivalent to
\begin{equation}
\beta=\beta_{\alpha}=1-\frac{1}{2^{\frac{1}{\alpha+1}}}.\label{eq:beta_alpha}
\end{equation}
\end{proof}
\smallskip


\section{Proof of Proposition \ref{prop:embed} and Theorem \ref{thm1}}
\label{pr21}

\subsection{Proof of Proposition \ref{prop:embed}}

The statement for $\alpha>p-1$ is trivial. Indeed, by the standard
Cauchy formula, 
\[
f'(u)=\bigg<f,\,\frac{z}{(1-\overline{u}z)^{2}}\bigg>,\qquad u\in\mathbb{D},
\]
and thus, bounding $\vert f'(u)\vert$ from above 
by $\left\Vert f\right\Vert _{BMOA}
\left\Vert \frac{z}{(1-\bar{u}z)^2}\right\Vert _{H^{1}} = 
(1-|u|^2)^{-1}\left\Vert f\right\Vert _{BMOA}$,
we get 
\[
\left\Vert f'\right\Vert _{A_{p}\left(\alpha\right)}^{p}\lesssim\left\Vert f\right\Vert _{BMOA}^{p}\int_{\mathbb{D}}(1-\vert u\vert)^{\alpha-p}{\rm d}A(u)\lesssim\left\Vert f\right\Vert _{BMOA}^{p}.
\]

For $p\ge 2$ and $\alpha=p-1$ we have
\begin{eqnarray*}
\left\Vert f'\right\Vert _{A_{p}\left(\alpha\right)}^{p} 
& = & \int_{\mathbb{D}}(1-|u|^2) \left|f'(u)\right|^p {\rm d}A(u)\\
& \leq & \left\Vert f\right\Vert _{B_{\infty}}^{p-2}\int_{\mathbb{D}}(1-|u|^2)\left|f'(u)\right|^{2}{\rm d}A(u),
\end{eqnarray*}
where $\left\Vert f\right\Vert _{B_{\infty}}$ is the norm of $f$
in the Bloch space. Since $\int_{\mathbb{D}}(1-|u|^2)\left|f'(u)\right|^{2}{\rm d}A(u)\le
\left\Vert f\right\Vert _{H^{2}}^{2}$,
$\left\Vert f\right\Vert _{H^{2}}\lesssim\left\Vert f\right\Vert _{BMOA}$
and $\left\Vert f\right\Vert _{B_{\infty}}\lesssim\left\Vert f\right\Vert _{BMOA},$
we conclude that 
\begin{equation}
\label{bab}
\left\Vert f'\right\Vert _{A_{p}\left(\alpha\right)}\lesssim\left\Vert f\right\Vert _{BMOA}.
\end{equation}

Now we turn to the necessity of the restrictions on $p$ and $\alpha$ for the estimate
\eqref{bab}. If $\alpha<p-1$ then it is well known that there exist interpolating 
Blaschke products $B$ such that $B' \notin A_{p}\left(\alpha\right)$ 
(see, e.g., \cite[Theorem 6]{Glu}, where an explicit criterion for the inclusion
is given in terms of the zeros of $B$). Finally,  
by a result of S.A. Vinogradov \cite[Lemma 1.6]{vin},
if $f\in A_{p}^{1}(p-1)$, $1\le p<2$, then, 
$\sum_{n=0}^{\infty}\vert\hat{f}(2^{n})\vert^{p}<\infty$ (where
$\hat{f}(n)$ stands for the $n^{th}$ Taylor coefficient of $f$).
Hence, $A_{p}^{1}(p-1)$ does not contain even some functions from the disc algebra, and so
$BMOA\nsubseteq A_{p}^{1}(p-1)$ when $1\le p<2$. 
\qed


\subsection{Integral representation for the derivative of functions in $K_{\Theta}$}

An important ingredient of our proof is the following simple and well-known
integral representation for the derivative of a function from a model
space.

\begin{lem}
\label{lemma_int_rep_KTheta} Let $\Theta$ be an inner function,
$f\in K_{\Theta}$, $n\in \mathbb{N}$. We have
\[
f^{(n)}(u)=\left\langle f,\, z^n \left(k_{u}^{\Theta}\right)^{n+1}\right\rangle,
\]
for any $u\in\mathbb{D}.$
\end{lem}

\begin{proof}
For a fixed $u\in\mathbb{D}$, we have
\[
f^{(n)}(u)=\bigg<f,\,\frac{z^n}{(1-\overline{u}z)^{n+1}}\bigg>=
\big<f,\, z^n \left(k_{u}^{\Theta}\right)^{n+1}\big>.
\]
Here the first equality is the standard Cauchy formula, while the
second follows from the fact that $z^n (1-\overline{u}z)^{-n-1}-z^n
\left(k_{u}^{\Theta}(z)\right)^{n+1}\in\Theta H^{2}$
and $f\perp \Theta H^{2}$.
\end{proof}


\subsection{Proof of the left-hand side inequality in (\ref{norm_D_K_Theta}) }

Sufficiency of the condition $\Theta' \in A_{p}\left(\alpha\right)$ in Theorem \ref{thm1}
and the left-hand side inequality in (\ref{norm_D_K_Theta}) follow immediately
from Theorem \ref{th3} and the following proposition.

\begin{prop}
\label{ineq_diff_KTheta_gen} Let $\alpha>-1$ and $1<p < \infty$,
let $\Theta$ be an inner function and $f\in K_{\Theta}.$ Then we have
\[
\left\Vert f'\right\Vert _{A_{p}\left(\alpha\right)}\leq\left\Vert f\right\Vert _{BMOA}\left(I_{p,\alpha}(\Theta)\right)^{\frac{1}{p}}.
\]
\end{prop}

\begin{proof}
We use the integral representation of $f'$ from Lemma \ref{lemma_int_rep_KTheta}:
\[
f'(u)=\left\langle f,\, z\left(k_{u}^{\Theta}\right)^{2}\right\rangle =\int_{\mathbb{T}}f(\tau)\overline{\tau\left(k_{u}^{\Theta}(\tau)\right)^{2}}{\rm d}m(\tau),
\]
for any $u\in\mathbb{D},$ and thus
\begin{eqnarray*}
\left\Vert f'\right\Vert _{A_{p}\left(\alpha\right)}^{p} & = & \int_{\mathbb{D}}(1-\vert u\vert^{2})^{\alpha}\left|\int_{\mathbb{T}}f(\tau)\overline{\tau\left(k_{u}^{\Theta}(\tau)\right)^{2}}{\rm d}m(\tau)\right|^{p}{\rm d}A(u)\\
 & \leq & \left\Vert f\right\Vert _{BMOA}^{p}\int_{\mathbb{D}}(1-\vert u\vert^{2})^{\alpha}\left(\int_{\mathbb{T}}\left|k_{u}^{\Theta}(\tau)\right|^{2}{\rm d}m(\tau)\right)^{p}{\rm d}A(u)\\
 & = & \left\Vert f\right\Vert _{BMOA}^{p}\int_{\mathbb{D}}(1-\vert u\vert^{2})^{\alpha}\left(\frac{1-\vert\Theta(u)\vert^{2}}{1-\left|u\right|^{2}}\right)^{p}{\rm d}A(u),
\end{eqnarray*}
which completes the proof. 
\end{proof} 

It remains to combine Proposition \ref{ineq_diff_KTheta_gen} 
with Theorem \ref{th3} to complete the
proof of the left-hand side inequality in \eqref{norm_D_K_Theta}.


\subsection{Proof of the right-hand side inequality in (\ref{norm_D_K_Theta}) }

To prove the necessity of the inclusion $\Theta' \in A_{p}\left(\alpha\right)$ 
and the left-hand side inequality in (\ref{norm_D_K_Theta}), consider the test function 
\[
f=S^{\star}\Theta=\frac{\Theta-\Theta(0)}{z},
\]
where $S^{\star}$ is the backward shift operator (\ref{diff_and_shift_op}).
It is well-known that $f$ belongs to $K_{\Theta}$ and easy to check
that $\left\Vert f\right\Vert _{BMOA}\leq2$, whence
\begin{equation}
\label{ebn}
\left\Vert D\right\Vert _{\left(K_{\Theta},\,\left\Vert 
\cdot\right\Vert _{BMOA}\right)\rightarrow A_{p}(\alpha)}
\geq\frac{\left\Vert f'\right\Vert _{A_{p}\left(\alpha\right)}}{2}.
\end{equation}
Now, 
\[
\left\Vert f'\right\Vert _{A_{p}\left(\alpha\right)}^{p}\geq\int_{\beta_{\alpha}}^{1}r\, w(r)\int_{\mathbb{T}}\left|f'\left(r\xi\right)\right|^{p}{\rm d}m(\xi)\,{\rm d}r,
\]
where $\beta_{\alpha}$ is given by (\ref{eq:beta_alpha}) and thus,
\begin{align*}
\left\Vert f'\right\Vert _{A_{p}\left(\alpha\right)} & \geq\left(\int_{\beta_{\alpha}}^{1}r\, w(r)\int_{\mathbb{T}}\left|\frac{\Theta'\left(r\xi\right)}{r\xi}\right|^{p}{\rm d}m(\xi)\,{\rm d}r\right)^{\frac{1}{p}}\\
 & -\left(\int_{\beta_{\alpha}}^{1}r\, w(r)\int_{\mathbb{T}}\left|\frac{\Theta(r\xi)-\Theta(0)}{r^{2}\xi^{2}}\right|^{p}{\rm d}m(\xi)\,{\rm d}r\right)^{\frac{1}{p}}.
\end{align*}
On one hand, applying 
Lemma \ref{weighted_bergman_norm} with $w=w_{\alpha}$
and $\beta=\beta_{\alpha}$ we obtain 
\begin{eqnarray*}
\int_{\beta_{\alpha}}^{1}r\, w_{\alpha}(r)\int_{\mathbb{T}}\left|\frac{\Theta'\left(r\xi\right)}{r\xi}\right|^{p}{\rm d}m(\xi)\,{\rm d}r & \geq & \int_{0}^{2\pi}\int_{\beta_{\alpha}}^{1}rw_{\alpha}(r)\vert\Theta'(re^{i\theta})\vert^{p}{\rm d}r\frac{{\rm d}\theta}{\pi}\\
 & \geq & \frac{\beta_{\alpha}}{2}\int_{0}^{2\pi}\int_{0}^{1}rw_{\alpha}(r)\vert\Theta'(re^{i\theta})\vert^{p}{\rm d}r\frac{{\rm d}\theta}{\pi}=\frac{\beta_{\alpha}}{2}\parallel\Theta'\parallel_{A_{p}(\alpha)}^{p}.
\end{eqnarray*}
On the other hand, since $\|f\|_{H^\infty} \le 2$, we have
\[
\int_{\beta_{\alpha}}^{1}r\, w_{\alpha}(r)\int_{\mathbb{T}}
\left|\frac{\Theta(r\xi)-\Theta(0)}{r^{2}\xi^{2}}\right|^{p}{\rm d}m(\xi)\,{\rm d}r
\leq 2^{p}\int_{\beta_{\alpha}}^{1}\frac{w_{\alpha}(r)}{r^{p-1}}
{\rm d}r\leq\frac{2^{p}}{\beta_\alpha^{p-1}} 
\int_{\beta_{\alpha}}^{1}w_{\alpha}(r){\rm d}r.
\]
Finally, we conclude that 
\[
\left\Vert f'\right\Vert _{A_{p}\left(\alpha\right)}\geq
\Big(\frac{\beta_\alpha}{2}\Big)^{\frac{1}{p}} 
\|\Theta'\|_{A_{p} \left(\alpha\right)}-
2 \beta_\alpha^{\frac{1}{p} -1}
\left(\int_{\beta_{\alpha}}^{1}w_{\alpha}(r){\rm d}r\right)^{\frac{1}{p}},
\]
which, combined with \eqref{ebn}, gives us 
the right-hand side inequality in \eqref{norm_D_K_Theta}.
\qed


\subsection{Proof of Theorem \ref{thm1}}

To complete the proof of Theorem \ref{thm1}, we need to recall
the following theorem proved by Ahern \cite{Ahe1} for the case $1\leq p\leq2$
and generalized by Verbitsky \cite{Ver} and Gluchoff \cite{Glu}
to the range $1\leq p<\infty$. This theorem characterizes inner functions
$\Theta$ whose derivative belong to $A_p(\alpha)$.

\begin{thm*} \textup(\cite{Glu}\textup)
\label{thm-Gluchoff}  Let $\Theta$ be an inner function, $1\leq p<\infty,$
and $\alpha>-1.$

$(i)$ If $\alpha>p-1,$ then $\Theta'\in A_{p}(\alpha)$.

$(ii)$ If $p-2<\alpha<p-1$, then \textup{$\Theta'\in A_{p}(\alpha)$}
if and only if $\Theta'\in A_{1}\left(\alpha-p+1\right).$

$(iii)$ If $\alpha<p-2$ and $p>1$, then \textup{$\Theta'\in A_{p}(\alpha)$}
if and only if $\Theta$ is a finite Blaschke product.
\end{thm*}
\noindent 
{\it Proof of Theorem \ref{thm1}.}
Statement (a) is contained in Proposition \ref{prop:embed}.
In order to prove (b) and (c) of Theorem \ref{thm1}, 
we first remark that for $\alpha<p-1,$ it follows
from (\ref{norm_D_K_Theta}) that $D:\:\left(K_{\Theta},\,\left\Vert 
\cdot\right\Vert _{BMOA}\right)\rightarrow A_{p}\left(\alpha\right)$
is bounded if and only if $\Theta'\in A_{p}\left(\alpha\right).$
A direct application of the above Ahern--Verbitsky--Gluchoff theorem
completes the proof for $\alpha>p-2$. The case $\alpha =  p-2$
follows from the Arazy--Fisher--Peetre inequality \eqref{eq:AFP_inequ}. 
\qed
\medskip


\section{Proof of Peller type inequalities}
\label{peller}

In this section we prove Theorem \ref{thm2}.
From now on the inner
function $\Theta$ is a finite Blaschke product. Recall that if $f\in\mathcal{R}_{n}^{+}$
and $1/\overline{\lambda}_{1},\dots,1/\overline{\lambda}_{n}$ are
the poles of $f$ (repeated according to multiplicities), then $f\in K_{zB_{\sigma}}$
with $\sigma=(\lambda_{1},\dots,\lambda_{n})$.

We start with the proof of the upper bound 
in the Arazy--Fisher--Peetre inequality  \eqref{eq:AFP_inequ}.

\begin{lem}
\label{bab31}                                                                                         
Let $B$ be a finite Blaschke product with the zeros $\{z_j\}_{j=1}^n$. Then 
$$
|B''(u)| \le \sum_{j=1}^n \frac{1-|z_j|^2}{|1 - \overline{z_j} u|^3} +
\bigg(\frac{1-|B(u)|}{1-|u|}\bigg)^2, \qquad u\in\mathbb{D}.
$$
\end{lem}

\begin{proof}
Let $B = \prod_{j=1}^n b_{z_j}$, 
where $b_\lambda = \frac{|\lambda|}{\lambda}\cdot 
\frac{\lambda-z}{1-\overline{\lambda} z}$. 
Then it is easy to see that
\begin{equation}
\label{xax2}
|B''(u)| \le \sum_{j=1}^n \frac{1-|z_j|^2}{|1-\overline{z_j}u|^3} 
\bigg|\frac{B(u)}{b_{z_j}(u)}\bigg|+ 
2\sum_{1\le j< k\le n} \frac{1-|z_j|^2}{|1-\overline{z_j}u|^2} 
\frac{1-|z_k|^2}{|1-\overline{z_k}u|^2} \bigg|\frac{B(u)}{b_{z_j}(u) b_{z_k}(u)}\bigg|,
\end{equation}

To estimate the second sum in \eqref{xax2},
first we note that $\frac{1-|\lambda|^2}{2|1-\overline{\lambda}{u}|^2} \le
\frac{1-|b_\lambda(u)|}{1-|u|} \le
\frac{2(1-|\lambda|^2)}{|1-\overline{\lambda}{u}|^2}$.
Let us introduce the notations $B_j = \prod_{l=1}^{j-1}b_{z_l}$ (assuming $B_1\equiv 1$)
and $\widehat{B}_k = \prod_{l=k}^n b_{z_l}$. 
Then 
\begin{equation}
\label{uss5}
\frac{1-|B(u)|}{1-|u|} \asymp \sum_{j=1}^n |B_j(u)|
\frac{1-|z_j|^2}{|1-\overline{z_j}u|^2} \asymp
\sum_{k=1}^n |\widehat{B}_{k+1}(u)|
\frac{1-|z_k|^2}{|1-\overline{z_k}u|^2}.
\end{equation}
It follows from the estimate $|B/(b_{z_j}b_{z_k})|\le |B_j \widehat{B}_{k+1}|$ and \eqref{uss5}
that
$$
\sum_{1\le j< k\le n} \frac{1-|z_j|^2}{|1-\overline{z_j}u|^2} 
\frac{1-|z_k|^2}{|1-\overline{z_k}u|^2} \bigg|\frac{B(u)}{b_{z_j}(u) b_{z_k}(u)}\bigg| 
\le \bigg( \sum_{j=1}^n |B_j(u)|
\frac{1-|z_j|^2}{|1-\overline{z_j}u|^2} \bigg) \times
$$
$$
\times \bigg(\sum_{k=1}^n |\widehat{B}_{k+1}(u)|
\frac{1-|z_k|^2}{|1-\overline{z_k}u|^2}\bigg) \lesssim
\bigg(\frac{1-|B(u)|}{1-|u|}\bigg)^2.
$$     
\end{proof}
            
Using Lemma \ref{bab31}, we first obtain the Arazy--Fisher--Peetre inequality for $p=1$:
$$
\|B\|_{B_1} \lesssim 
\int_{\mathbb{D}} 
|B''(u)|{\rm d}A(u) \lesssim 
\sum_{j=1}^n \int_{\mathbb{D}}
\frac{1-|z_j|^2}{|1-\overline{z_j}u|^3}
{\rm d}A(u) + I_{2,0}(B) \lesssim n.
$$
We used Dyn'kin's inequality \eqref{Dyn'kin_estimate}  
and the fact that, by the well-known  estimates (see \cite[Theorem 1.7]{HKZ}), 
each integral in the above sum does not exceed some absolute constant, which
does not depend on $z_j$. Finally, for $1<p<\infty$, we have
$$
\begin{aligned}
\|B\|_{B_{p}}^{\star\;p} & \asymp 
\int_{\mathbb{D}} 
|B''(u)|^p (1-|u|^2)^{2p-2}{\rm d}A(u) \\
& \le 
\big(\sup_{u\in\mathbb{D}}|B''(u)|(1-|u|^2)^2\big)^{p-1}\int_{\mathbb{D}} 
|B''(u)|{\rm d}A(u) \lesssim n,
\end{aligned}
$$
since $\sup_{u\in\mathbb{D}}|f''(u)|(1-|u|)^2 \le 2\|f\|_{H^\infty}$.    
\bigskip
\\
{\it Proof of Theorem \ref{thm2}.}
Let $f\in\mathcal{R}_{n}^{+}$;
there exists $\sigma\in\mathbb{D}^{n}$ such that $f\in K_{\widetilde{B}_{\sigma}},$
$\widetilde{B}_{\sigma}=zB_{\sigma}.$ Then, by Proposition \ref{ineq_diff_KTheta_gen}
we have
\[
\left\Vert f'\right\Vert _{A_{p}\left(\alpha\right)}\leq
\left\Vert f\right\Vert _{BMOA}\left(I_{p,\alpha}(\widetilde{B}_{\sigma})\right)^{\frac{1}{p}}
\]
for any $\alpha>-1$ and $1<p<\infty$. Now applying Theorem \ref{th3},
we obtain
$\left\Vert f\right\Vert _{A_{p}^{1}(\alpha)}^{\star}\leq K_{p,\,\alpha}
\left\Vert f\right\Vert _{BMOA}\|B_{\sigma}'\|_{A_{p}\left(\alpha\right)}$.
Finally, note that
\[
\|\widetilde{B}_{\sigma}'\|_{A_{p}\left(\alpha\right)}
\leq\| zB_{\sigma}'\|_{A_{p}\left(\alpha\right)}+\| B_{\sigma}\|_{A_{p}
\left(\alpha\right)}\lesssim\| B_{\sigma}'\|_{A_{p}\left(\alpha\right)}.
\]
\qed

\begin{rem*}
In Subsection \ref{Subsec_Peller_gen}, we have shown how to deduce
Peller's inequality (\ref{Peller_estimate}) from Theorem \ref{thm2} and 
the result of Arazy--Fischer--Peetre \eqref{eq:AFP_inequ}. 
Let us show that for $p\geq2$ one can give
a very simple proof which uses only Proposition \ref{ineq_diff_KTheta_gen}
and Dyn'kin's estimate $I_{2,\,0}(\widetilde{B}_{\sigma})\le8(n+2),$
where $n={\rm deg}\, B_{\sigma}.$ Indeed, in this case, we have
\begin{eqnarray*}
I_{p,\, p-2}(\widetilde{B}_{\sigma}) & = & \int_{\mathbb{D}}(1-\vert u\vert^{2})^{p-2}\left(\frac{1-\vert\widetilde{B}_{\sigma}(u)\vert^{2}}{1-\left|u\right|^{2}}\right)^{p-2+2}{\rm d}A(u)\\
 & = & \int_{\mathbb{D}}\left(1-\vert\widetilde{B}_{\sigma}(u)\vert^{2}\right)^{p-2}\left(\frac{1-\vert\widetilde{B}_{\sigma}(u)\vert^{2}}{1-\left|u\right|^{2}}\right)^{2}{\rm d}A(u)
\leq  I_{2,\,0}(\widetilde{B}_{\sigma}).
\end{eqnarray*}
It remains to apply Proposition \ref{ineq_diff_KTheta_gen} with $\alpha=p-2$.                                                                
\end{rem*}
\medskip


\section{An elementary proof of Peller's inequality for $p>\frac{1}{2}$}
\label{pel1}

In this section we prove the inequality
\begin{equation}
\label{p1}
\|f\|_{B_{p}} \leq c n^{\frac{1}{p}} \left\Vert f\right\Vert _{BMOA}
\end{equation}
for $1\ge p >\frac{1}{2}$ using the integral representations of the derivatives 
in model spaces. It is well known and easy to see that, for $p>\frac{1}{2}$,
\[
\left\Vert f\right\Vert _{B_{p}}^p 
\asymp |f(0)|^p + |f'(0)|^p + |f''(0)|^p +
\int_{\mathbb{D}} |f'''(u)|^p (1-|u|^2)^{3p-2} {\rm d}A(u).
\]
Thus in what follows it is the last integral (which we denote
$\|f\|_{B_p}^{\star\star}$) that we will estimate.

Let $\Theta$ be an inner function and let $f\in K_\Theta$.
Then, by Lemma \ref{lemma_int_rep_KTheta},
$|f'''(u)|=\big|\big\langle f,\, z^3 (k_{u}^{\Theta})^{4}\big\rangle\big| \le  
\|k_u^\Theta\|_4^4 \|f\|_{BMOA}$,
and so
$$
\left\Vert f\right\Vert _{B_{p}}^{\star\star} \le
\|f\|_{BMOA}^p \int_{\mathbb{D}} \|k_u^\Theta\|_4^{4p}(1-|u|^2)^{3p-2} {\rm d}A(u).
$$

\begin{lem}
For any $u\in\mathbb{D}$,
$$
\|k_u^\Theta\|_4^4 = \frac{(1+|u|^2)(1-|\Theta(u)|^4)}{(1-|u|^2)^3} -
\frac{4 {\rm Re}\, (u\Theta'(u) \overline{\Theta(u)})}{(1-|u|^2)^2}.
$$
\end{lem}

\begin{proof} The lemma follows from  straightforward computations based on the formula
$f'(u)= \langle f,\, z (1-\overline{u}z)^{-2}\rangle$. We omit the details.
\end{proof}

We continue to estimate $\| f \|_{B_p}^{\star\star}$.
Since $1+|u|^2 \le 2$ and $1-|\Theta(u)|^4 \le 2(1-|\Theta(u)|^2)$, we have
$$
\|k_u^\Theta\|_4^4 \le \frac{4(1-|\Theta(u)|^2)}{(1-|u|^2)^3} -
\frac{4 {\rm Re}\, (u\Theta'(u) \overline{\Theta(u)})}{(1-|u|^2)^2}.
$$
From now on assume that $\Theta$ is a finite Blaschke product $B =
\prod_{k=1}^n b_{z_k}$. Then
$$
\begin{aligned}
uB'(u)\overline{B(u)} & = u|B(u)|^2\frac{B'(u)}{B(u)} =
u|B(u)|^2 \sum_{k=1}^n \bigg(\frac{1}{u-z_k} +\frac{\overline{z_k}}{1-\overline{z_k}u}
\bigg) \\
& = |B(u)|^2\sum_{k=1}^n \frac{1-|z_k|^2}{|1-\overline{z_k}u|^2} +
|B(u)|^2\sum_{k=1}^n \frac{z_k(1-|z_k|^2)(1-|u|^2)}{|1-\overline{z_k}u|^2(u-z_k)}.
\end{aligned}
$$
Denote the last term by $S_1(u)$. Since $|B(u)|\le |b_{z_k}(u)|$, we have
$$
|S_1(u)| \le \sum_{k=1}^n \frac{(1-|u|^2)(1-|z_k|^2)}{|1-\overline{z_k}u|^3},
$$
whence (recall that $p\le 1$)
$$
\int_{\mathbb{D}} \frac{|S_1(u)|^p}{(1-|u|^2)^p} (1-|u|^2)^{3p-2} {\rm d}A(u) \le
\sum_{k=1}^n 
\int_{\mathbb{D}} \frac{(1-|z_k|^2)^p}{|1-\overline{z_k}u|^{3p}}
(1-|u|^2)^{3p-2} {\rm d}A(u) \lesssim n, 
$$
since, by \cite[Theorem 1.7]{HKZ}, 
each integral in the above sum does not exceed some constant depending only on $p$, but not 
on $z_k$.

Thus, to prove \eqref{p1}, it remains to estimate the weighted area integral of the difference
$$
S_2(u) = \frac{1-|B(u)|^2}{(1-|u|^2)^2} -
\frac{|B(u)|^2}{1-|u|^2} \sum_{k=1}^n \frac{1-|z_k|^2}{|1-\overline{z_k}u|^2}.
$$
We use again the notations $B_k = \prod_{l=1}^{k-1}b_{z_l}$ (assuming $B_1\equiv 1$)
and $\widehat{B}_k = \prod_{l=k}^n b_{z_l}$. It is easy to see that
\begin{equation}
\label{uss}
\frac{1-|B(u)|^2}{1-|u|^2} = \sum_{k=1}^n |B_k(u)|^{2}
\frac{1-|z_k|^2}{|1-\overline{z_k}u|^2}.
\end{equation}
Hence,
$$
\begin{aligned}
S_2(u) & = \sum_{k=1}^n |B_k(u)|^{2} \cdot
\frac{1-|z_k|^2}{|1-\overline{z_k}u|^2} \cdot
\frac{1-|\widehat{B}_k(u)|^2}{1-|u|^2} \\
& =
\sum_{k=1}^n \sum_{l=k}^{n}
\frac{1-|z_k|^2}{|1-\overline{z_k}u|^2} \cdot
\frac{1-|z_l|^2}{|1-\overline{z_l}u|^2} \cdot|B_l(u)|^{2}.
\end{aligned}
$$
Note that, by a formula analogous to \eqref{uss}, but without squares,
$$
\frac{1-|B(u)|}{1-|u|} =
\sum_{k=1}^n |B_k(u)|
\frac{1-|b_{z_k}(u)|}{1-|u|}\ge \frac{1}{2}\sum_{k=1}^n |B_k(u)|
\frac{1-|z_k|^2}{|1-\overline{z_k}u|^2}.
$$
Hence,
$$
4\bigg(\frac{1-|B(u)|}{1-|u|}\bigg)^2 = \sum_{k=1}^n \sum_{l=1}^n
|B_k(u)|\cdot |B_l(u)|\cdot
\frac{1-|z_k|^2}{|1-\overline{z_k}u|^2} \cdot
\frac{1-|z_l|^2}{|1-\overline{z_l}u|^2}.
$$
Denote the last double sum by $S_3(u)$. Since $|B_kB_l|\ge |B_l|^2$, $l\ge k$, we
see that $S_2(u) \le S_3(u)$. Now we have 
$$
\begin{aligned}
\int_{\mathbb{D}} \frac{|S_2(u)|^p}{(1-|u|^2)^p} (1-|u|^2)^{3p-2} {\rm d}A(u) & \le
4^p \int_{\mathbb{D}} \bigg(\frac{1-|B(u)|}{1-|u|}\bigg)^{2p}  
(1-|u|^2)^{2p-2} {\rm d}A(u) \\
&  \lesssim
I_{2p, 2p-2}(B)\lesssim \|B'\|^{2p}_{B_{2p}} \lesssim n.
\end{aligned}
$$
Here we used Theorem \ref{th3} to estimate $I_{2p, 2p-2}(B)$ (recall that $2p>1$) 
and the Arazy--Fisher--Peetre inequality   \eqref{eq:AFP_inequ}.
\qed
\medskip


\section{\label{sec:Radial-weighted-Bergman-norms}Radial-weighted Bergman
norms of the derivative of finite Blaschke products}

Again, let $n\geq1,$ $\sigma=\left(\lambda_{1},\,\dots,\,\lambda_{n}\right)\in\mathbb{D}^{n}$
and let $B_{\sigma}$ be the finite Blaschke product corresponding to
$\sigma.$ For any $1<p<\infty$ and $\alpha>-1,$ we set
\[
\varphi_{n}(p,\,\alpha)=\sup\left\{ \left\Vert B_{\sigma}'\right\Vert _{A_{p}\left(\alpha\right)}:\:\sigma\in\mathbb{D}^{n}\right\} .
\]

Note that for any $n\geq1,$ $\varphi_{n}(p,\,\alpha)=\varphi_{1}(p,\,\alpha)=\infty$
if and only if $\alpha<p-2$. Indeed, if $\alpha \ge p-2$, then 
$\varphi_{n}(p,\,\alpha) \le \varphi_{n}(p,\,p-2) \asymp n^{\frac{1}{p}}$ by the 
Arazy--Fisher--Peetre inequality \eqref{eq:AFP_inequ}. For $\alpha<p-2$, 
consider the test function $b_{r}$, $r\in(0,1)$. It is easily seen
(see, e.g., \cite[Theorem 1.7]{HKZ}) 
that $\left\Vert b_{r}'\right\Vert _{A_{p}\left(\alpha\right)} \to \infty$ as $r \to 1-$. 

We have seen in Subsection \ref{Subsec_Peller_gen} how 
the estimate $\varphi_{n}(p,\, p-2)\asymp n^{\frac{1}{p}}$
implies Peller's inequality (\ref{Peller_estimate}). It could
be of interest to find a more general estimate (for other values of
$\alpha$ and $p$) of $\varphi_{n}(p,\,\alpha)$. Notice that for
each fixed $p,$ the function $\alpha\mapsto\varphi_{n}(p,\,\alpha)$
is decreasing and there exists the second critical value 
$\alpha_{p}\geq-1$,  
\[
\alpha_{p}=\inf\left\{ \alpha>-1:\:\sup_{n}\varphi_{n}(p,\,\alpha)<\infty\right\} .
\]
The sequence $\{\varphi_{n}(p,\,\alpha)\}_{n\geq1}$
may be unbounded and, thus, have a nontrivial asymptotics 
if and only if $p-2 \le \alpha\le \alpha_{p}$. In this notation we can rewrite Theorem
\ref{thm2} as 
\[
\left\Vert f\right\Vert _{A_{p}^{1}(\alpha)}^{\star}\lesssim\varphi_{n}(p,\,\alpha)\left\Vert f\right\Vert _{BMOA},
\]
for any $f\in\mathcal{R}_{n}^{+}$, $1<p<\infty$,
$p-2 \le \alpha\le \alpha_{p}$.

We will show now that $\alpha_p = p-1$, and so $p-1$ 
is the second critical value of $\alpha$, as is expected from Theorem \ref{thm1}.

\begin{prop} \label{Prop_critical_alpha}For any $p>1,$ $\alpha_{p}=p-1.$
\end{prop}

\begin{proof}
By the Schwarz--Pick lemma, we have that for any $\sigma\in\mathbb{D}^{n},$
\[
\left\Vert B_{\sigma}'\right\Vert _{A_{p}\left(\alpha\right)}^{p}\leq I_{p,\alpha}(B_{\sigma}),
\]
and for any $\alpha>p-1,$
\[
I_{p,\alpha}(B_{\sigma})\leq\int_{\mathbb{D}}(1-\vert u\vert^{2})^{\alpha}\left(\frac{1-\vert B_{\sigma}(u)\vert^{2}}{1-\left|u\right|^{2}}\right)^{p}{\rm d}A(u)\lesssim\int_{\mathbb{D}}\frac{1}{(1-\left|u\right|)^{p-\alpha}}{\rm d}A(u)<\infty,
\]
and thus $\alpha_{p}\leq p-1$ for each $p$.

Next we show that $\alpha_{p}\geq p-1.$ Let us consider the set
$\sigma=\left(0,\,\dots,\,0\right)\in\mathbb{D}^{n},$
for which $B_{\sigma}(z)=z^{n}.$ In this case, we have
\[
\left\Vert B_{\sigma}'\right\Vert _{A_{p}\left(\alpha\right)}^{p}
=\left\Vert nz^{n-1}\right\Vert_{A_p(\alpha)}^{p}=n^{p}\int_{0}^{1}r\,(1-r^{2})^{\alpha}\int_{\mathbb{T}}\left|r\xi\right|^{p(n-1)}{\rm d}m(\xi)\,{\rm d}r,
\]
which gives
\[
\left\Vert B_{\sigma}'\right\Vert _{A_{p}\left(\alpha\right)}^{p}=n^{p}\int_{0}^{1}(1-r^{2})^{\alpha}r^{p(n-1)+1}{\rm d}r,
\]
and
\[
\beta(pn-p+2,\,\alpha+1)\leq\frac{\left\Vert B_{\sigma}'\right\Vert _{A_{p}\left(\alpha\right)}^{p}}{n^{p}}\leq2^{\alpha}\beta(pn-p+2,\,\alpha+1),
\]
where $\beta$ stands for the \textit{Beta function} $\beta(x,\, y)=
\int_{0}^{1}r^{x-1}(1-r)^{y-1}{\rm d}r$.
Let $\alpha=p-1-\varepsilon$, $\varepsilon>0$. Then by the standard $\Gamma$-function
asymptotics, we obtain
\[
\left\Vert B_{\sigma}'\right\Vert _{A_{p}\left(\alpha\right)}^{p}\geq
\Gamma(\alpha+1)n^{p}\frac{\Gamma(pn-p+2)}{\Gamma(pn+\alpha-p+3)}
\sim_{n\rightarrow\infty}\Gamma(\alpha+1)n^{p}(pn)^{\varepsilon-p},
\]
whence $\sup_{n}\varphi_{p}(\alpha,\, n)=\infty.$
\end{proof}
\medskip


\section{Remarks on Dolzhenko's inequalities}
\label{dolg}

\subsection{Proof of Dolzhenko's inequalities for $1\le p\le 2$.}
In \cite[Theorem 2.2]{Dol} E.P.~Dolzhenko proved
that for any $f\in\mathcal{R}_{n}^{+}$,
\begin{equation}
\label{Dol1}
\left\Vert f'\right\Vert _{A_p}\lesssim
\begin{cases}
n^{1-\frac{1}{p}}\left
\Vert f\right\Vert _{H^{\infty}}, & 1<p\le 2, \\
\log n\left\Vert f\right\Vert _{H^{\infty}}, & p=1,
\end{cases}
\end{equation}
where the constants involved in $\lesssim$ may depend on $p$ only.
Let us show that these inequalities (and even with $BMOA$-norm
in place of $H^\infty$-norm)
are direct corollaries of Proposition~\ref{ineq_diff_KTheta_gen} and the following simple lemma.

\begin{lem}
For any Blaschke product $B$ of degree $n$ we have
\begin{equation}
\label{Dol2}
I_{p,0}(B) =
\int_{\mathbb{D}} \bigg(\frac{1-|B(u)|^2}{1-|u|^{2}} \bigg)^{p}{\rm d}A(u)
\lesssim
\begin{cases}
n^{p-1}, & 1<p\le 2, \\
\log n, & p=1.
\end{cases}
\end{equation}
\end{lem}

\begin{proof}
Clearly, the integral over the disc $\big\{|z| \le 1-\frac{1}{n} \big\}$
has the required estimate. The estimate over the annulus $\big\{1- \frac{1}{n}
\le |z| <1\big\}$
follows from the result of Dyn'kin ($I_{2,0}(B) \lesssim n$)
and the H\"older inequality. Indeed, for $1\le p <2$,
$$
\int_{\big\{1- \frac{1}{n}
\le |z| <1\big\}} \bigg(\frac{1-|B(u)|^2}{1-|u|^{2}} \bigg)^{p}{\rm d}A(u)  \le
(I_{2,0}(B))^{\frac{p}{2}} \Big(\pi(1-\Big(1-\frac{1}{n}\Big)^2\Big)^{1-\frac{p}{2}}
\lesssim n^{p-1}.
$$
\end{proof}

Now inequality \eqref{Dol1} follows from \eqref{Dol2} and from the inequality
$\|f'\|_{A_p(\alpha)} \le \|f\|_{BMOA} (I_{p,\alpha}(B_\sigma))^{\frac{1}{p}}$
which holds for any function $f\in K_{B_\sigma}$ 
(see Proposition~\ref{ineq_diff_KTheta_gen}).
It should be mentioned, however, that Dolzhenko
proves his inequalities for more general domains than the unit disc.


\subsection{An extension of Dolzhenko's inequalities to the range $p>2$.}

The case $p>2$ is also treated by Dolzhenko (see the last inequality
in \cite[Theorem 2.2]{Dol}), but the corresponding analog is of somewhat
different nature. As the example $f(z) = (1-rz)^{-1}$ with $r\to 1-$ shows,
there exist no estimate of $\|f'\|_{A^p}$ in terms of $\|f\|_{BMOA}$ and $n =
{\rm deg}\, f$.

Here we obtain another extension of Dolzhenko's result for $p>2$.

\begin{thm}
Let $2 < p\leq\infty$, let $f\in\mathcal{R}_{n}^{+}$, $n\ge 1$, and
let $1/\overline{\lambda}_{1},\dots,1/\overline{\lambda}_{n}$ be its
poles \textup(repeated according to multiplicities\textup). We have
\begin{equation}
\label{eq:dol_p_geq_2}
\left\Vert f'\right\Vert _{A_{p}}\lesssim n^{\frac{1}{p}}
\bigg(\sum_{k=1}^{n}
\frac{1+\vert\lambda_{k}\vert}{1-\vert\lambda_{k}\vert}\bigg)^{1-\frac{2}{p}}
\left\Vert f\right\Vert _{BMOA}.
\end{equation}

Moreover, the inequality \eqref{eq:dol_p_geq_2} is asymptotically
sharp in the following sense:
for any $r\in(0,\,1)$ there exists $g\in\mathcal{R}_{n}^{+}$
having $\frac{1}{r}$ as a pole of multiplicity $n$ such that
\begin{equation}
\label{eq:sharp_dol_p_geq_2}
\left\Vert g'\right\Vert _{A_{p}}\gtrsim  n^{1-\frac{1}{p}}
\left(\frac{1}{1-r}\right)^{1-\frac{2}{p}}
\left\Vert g\right\Vert _{BMOA}.
\end{equation}
\end{thm}

\begin{proof}
We first prove (\ref{eq:dol_p_geq_2}). Set $\sigma=(\lambda_{1},\dots,\lambda_{n})$,
so that $f\in K_{\widetilde{B}_{\sigma}}$. By Proposition 4.2 and
Dyn'kin's inequality (2.6),
\[
\left\Vert f'\right\Vert _{A_{p}}^{p}\leq\left\Vert f\right\Vert _{BMOA}^{p}\int_{\mathbb{D}}\left(\frac{1-\vert\widetilde{B}_{\sigma}(u)\vert^{2}}{1-\left|u\right|^{2}}\right)^{p}{\rm d}A(u)\lesssim n\left\Vert f\right\Vert _{BMOA}\sup_{u\in\mathbb{D}}\left(\frac{1-\vert\widetilde{B}_{\sigma}(u)\vert^{2}}{1-\left|u\right|^{2}}\right)^{p-2}.
\]
By \eqref{uss}, we have
\[
\frac{1-\vert\widetilde{B}_{\sigma}(u)\vert^{2}}{1-\left|u\right|^{2}}\leq1+\sum_{k=1}^{n}\frac{1-\vert\lambda_{k}\vert^{2}}{\vert1-\overline{\lambda_{k}}u\vert^{2}}\lesssim\sum_{k=1}^{n}\frac{1+\vert\lambda_{k}\vert}{1-\vert\lambda_{k}\vert}.
\]
This proves \eqref{eq:dol_p_geq_2}. Now we prove \eqref{eq:sharp_dol_p_geq_2}.
Take $g=b_{-r}^n(u)$, $r\in(0,1),$ then
\[
\left\Vert g'\right\Vert _{A_{p}}^{p}=n^{p}\int_{\mathbb{D}}\left|b_{-r}'(u)\right|^{2}\left|b_{-r}'(u)\right|^{p-2}\left|b_{-r}(u)\right|^{p(n-1)}{\rm d}A(u).
\]
Taking $v=b_{-r}(u)$ as the new variable 
and using the fact that $u=b_{-r}(v)$, we get
\[
\left\Vert g'\right\Vert _{A_{p}}^{p}=n^{p}\int_{\mathbb{D}}\left|b_{-r}'(b_{-r}(v))\right|^{p-2}\left|v\right|^{p(n-1)}{\rm d}A(v).
\]
Since $b_{-r}'\circ b_{-r} (v)=-\frac{(1+rv)^{2}}{1-r^{2}}$,
we obtain
\[
\left\Vert g'\right\Vert _{A_{p}}^{p}=
\frac{n^{p}}{(1-r^{2})^{p-2}}\int_{\mathbb{D}}\left|1+rv\right|^{2(p-2)}\left|v\right|^{p(n-1)}{\rm d}A(v).
\]
Supposing that $p\geq2,$ we have
\[
\int_{\mathbb{D}}\left|1+rv\right|^{2(p-2)}\left|v\right|^{p(n-1)}{\rm d}A(v)\asymp\int_{\mathbb{D}}\left|v\right|^{p(n-1)}{\rm d}A(v)
=\frac{2}{pn-p+2},
\]
whence
\[
\left\Vert g'\right\Vert _{A_{p}}^{p}\asymp\frac{n^{p-1}}{(1-r)^{p-2}}.
\]
Since $\|g\|_{BMOA} = 1$ this completes the proof \eqref{eq:sharp_dol_p_geq_2}.
\end{proof}

\end{document}